\documentclass{article}

\usepackage{amsmath,amssymb,amsthm,mathrsfs}
\usepackage{graphicx}

				\newtheorem{theorem}{Theorem}
                
                \newtheorem{lemma}[theorem]{Lemma}
\theoremstyle{definition} \newtheorem*{definition}{Definition}
\theoremstyle{definition} \newtheorem{hypothesis}{Hypothesis}
\theoremstyle{definition} \newtheorem{remark}{Remark}
\theoremstyle{definition} \newtheorem{example}{Example}

\begin{document}

\title{Stochastic Evolution Equations with Multiplicative Poisson Noise and Monotone Nonlinearity: A New Approach}
\author{Erfan Salavati, Bijan Z. Zangeneh\\ \\
Department of Mathematical Sciences,\\
Sharif University of Technology\\
Tehran, Iran}
\date{}

\maketitle

\begin{abstract}
Semilinear stochastic evolution equations with multiplicative Poisson noise and monotone nonlinear drift are considered. We do not impose coercivity conditions on coefficients. A novel method of proof for establishing existence and uniqueness of the mild solution is proposed. Examples on stochastic partial differential equations and stochastic delay differential equations are provided to demonstrate the theory developed. 
\end{abstract}

\vspace{2mm}
\noindent{Mathematics Subject Classification: 60H10, 60H15, 60G51, 47H05, 47J35.}
\vspace{2mm}

\noindent{Keywords: Stochastic Evolution Equation, Monotone Operator, L\'evy Noise, It\"o type inequality, Stochastic Convolution Integral.}

\section{Introduction}\label{section: introduction}

Consider the stochastic evolution equation
\begin{equation}\label{main_equation}
    dX_t=AX_t dt+f(t,X_t) dt + g(t,X_{t-})d W_t + \int_E k(t,\xi,X_{t-}) \tilde{N}(dt,d\xi),
\end{equation}
where $W_t$ is a cylindrical Wiener process and $\tilde{N}(dt,d\xi)$ is a compensated Poisson random measure. We assume $f$ is semimonotone and $g$ and $k$ are Lipschitz and have linear growth. In section~\ref{section: Problem} the assumptions on coefficients are stated precisely. The purpose of this article is to prove the existence and uniqueness of the solution of this equation.

The special cases of equation~\eqref{main_equation} have been studied by several authors. For the case that all the coefficients are Lipschitz see~\cite{DaPrato_Zabczyk_book} for Wiener noise and~\cite{Kotelenez-1984} for general martingale noise and~\cite{Peszat-Zabczyk},~\cite{Albeverio-Mandrekar-Rudiger-2009} and~\cite{Marinelli-Prevot-Rockner} for the case of jump noise.
In the non-Lipschitz case there are two main approaches. The first approach is the variational approach in which the coefficients satisfy certain monotonicity and coercivity properties. For this approach see~\cite{Pardoux},~\cite{Krylov-Rozovskii} and~\cite{Rockner} for Wiener noise,~\cite{Gyongy} for general martingales and~\cite{Brzezniak-Liu-Zhu} for L\'evy noise. The second approach is the semigroup approach to semilinear stochastic evolution equations with monotone drift. This approach has first appeared in deterministic context in~\cite{Browder} and~\cite{Kato} and has been extended to stochastic evolution equations in~\cite{Zangeneh-Thesis} and~\cite{Zangeneh-Paper}.
There are other works with this approach, e.g the exponential asymptotic stability of solutions in the case of Wiener noise has been studied in~\cite{Jahanipour-Zangeneh}, generalizing the previous results to stochastic functional evolution equations with coefficients depending on the past path of the solution is done in~\cite{Jahanipur-functional-stability}, the large deviation principle for the case of Wiener noise is studied in~\cite{Dadashi-Zangeneh}. A limiting problem of such equations arising from random motion of highly elastic strings has been considered in~\cite{Zamani-Zangeneh-random-motion}. Finally, the stationarity of a mild solution to a stochastic evolution equation with a monotone nonlinear drift and Wiener noise is studied in~\cite{Zangeneh_Nualart}.

We should mention the remarkable article~\cite{Marinelli-Rockner-wellposedness} which considers monotone nonlinear drift and multiplicative Poisson noise on certain function spaces and proves the existence, uniqueness and regular dependence of the mild solution on initial data. They impose an additional positivity assumption on the semigroup and the drift term is the Nemitsky operator associated with a real monotone function. Their idea is to regularize the monotone nonlinearity $f$ by its Yosida approximation $f_\lambda(x)= \lambda^{-1}(x-(I+\lambda f)^{-1}(x))$. We will treat their result as a special case of our theory in Example~\ref{example: finite_diemnsional_noise}.

The semigroup approach to semilinear stochastic evolution equations with monotone nonlinearities has an advantage relative to the variational method since it does not require the coercivity. There are important examples, such as stochastic partial differential equations of hyperbolic type with monotone nonlinear terms, for which the generator does not satisfy the coercivity property and hence the variational method is not directly applicable to these equations. Pardoux~\cite{Pardoux} has developed a new theory for the application of the variational method to second order hyperbolic equations. But as is shown in Example~\ref{example: finite_diemnsional_noise_hyperbolic}, this problem can be treated directly in semigroup setting.

The main contribution of this article is Theorem~\ref{theorem:existence and uniqueness} in section~\ref{section: Existence and Uniqueness} which shows the existence and uniqueness of the mild solution for equation~\eqref{main_equation}. In section~\ref{section: Problem} the precise assumptions on coefficients are stated. In section~\ref{section:examples} we will provide some concrete examples to which our results apply. These examples consist of semilinear stochastic partial differential equations and a stochastic delay differential equation.

We will use the notion of stochastic integration with respect to cylindrical Wiener process and compensated Poisson random measure. For this definition and properties see~\cite{Peszat-Zabczyk} and~\cite{Albeverio-Mandrekar-Rudiger-2009}.

\section{The Assumptions}\label{section: Problem}

Let $H$ be a separable Hilbert space with inner product $\langle \, , \, \rangle$. Let $S_t$ be a $C_0$ semigroup on $H$ with infinitesimal generator $A:D(A)\to H$. Furthermore we assume the exponential growth condition on $S_t$ holds, i.e. there exists a constant $\alpha$ such that $\| S_t \| \le e^{\alpha t}$. If $\alpha=0$, $S_t$ is called a contraction semigroup. We denote by $L_{HS}(K,H)$ the space of Hilbert-Schmidt mappings from a Hilbert space $K$ to $H$.

Let $(\Omega,\mathcal{F},\mathcal{F}_t,\mathbb{P})$ be a filtered probability space. Let $(E,\mathcal{E})$ be a measurable space and $N(dt,d\xi)$ a Poisson random measure on $\mathbb{R}^+ \times E$ with intensity measure $dt \nu(d\xi)$. Our goal is to study equation~\eqref{main_equation} in $H$, where $W_t$ is a cylindrical Wiener process on a Hilbert space $K$ and $\tilde{N}(dt,d\xi)=N(dt,d\xi)-dt\nu(d\xi)$ is the compensated Poisson random measure corresponding to $N$. We assume that $N$ and $W_t$ are independent. We also assume the following,

\begin{hypothesis}\label{main_hypothesis}
    \begin{description}

        \item[(a)] $f(t,x,\omega):\mathbb{R}^+\times H\times \Omega \to H$ is measurable, $\mathcal{F}_t$-adapted, demicontinuous with respect to $x$ and there exists a constant $M$ such that
            \[ \langle f(t,x,\omega)-f(t,y,\omega),x-y \rangle \le M \|x-y\|^2,\]

        \item[(b)] $g(t,x,\omega):\mathbb{R}^+\times H\times \Omega \to L_{HS}(K,H)$ and $k(t,\xi,x,\omega):\mathbb{R}^+\times E\times H\times \Omega \to H$ are predictable and there exists a constant $C$ such that
            \[ \| g(t,x,\omega)-g(t,y,\omega)\|_{L_{HS}(K,H)}^2 + \int_{E}\|k(t,\xi,x)-k(t,\xi,y)\|^2 \nu(d\xi) \le C \|x-y \|^2,\]

        \item[(c)] There exists a constant $D$ such that
            \[ \| f(t,x,\omega)\|^2 + \| g(t,x,\omega)\|_{L_{HS}(K,H)}^2 + \int_{E}\|k(t,\xi,x)\|^2 \nu(d\xi) \le D(1+\|x\|^2),\]
        \item[(d)] $X_0(\omega)$ is $\mathcal{F}_0$ measurable and square integrable.
    \end{description}

\end{hypothesis}

\begin{definition}
    By a \emph{mild solution} of equation~\eqref{main_equation} with initial condition $X_0$ we mean an adapted c\`adl\`ag process $X_t$ that satisfies
    \begin{multline}\label{mild_solution}
        X_t=S_t X_0+\int_0^t S_{t-s}f(s,X_s) ds+\int_0^t{S_{t-s}g(s,X_{s-})d W_s}\\
        + \int_0^t{\int_E {S_{t-s}k(s,\xi,X_{s-})} \tilde{N}(ds,d\xi).}
    \end{multline}
\end{definition}

Because of the presence of monotone nonlinearity in our equation, the usual inequalities for stochastic convolution integrals are not applicable to equation~\eqref{main_equation}. For this reason we state the following inequality.

\begin{theorem}[It\^o type inequality, Zangeneh~\cite{Zangeneh-Paper}]\label{theorem:ito type inequality}
    Let $Z_t$ be an $H$-valued c\`adl\`ag locally square integrable semimartingale. If
    \[ X_t=S_t X_0 + \int_0^t S_{t-s}dZ_s, \]
    then
    \begin{equation*}
        \lVert X_t \rVert ^2 \le e^{2\alpha t}\lVert X_0 \rVert ^2 + 2 \int_0^t {e^{2\alpha (t-s)}\langle X_{s-} , d Z_s \rangle}+\int_0^t {e^{2\alpha (t-s)}d[Z]_s},
    \end{equation*}
    where $[Z]_t$ is the quadratic variation process of $Z_t$.
\end{theorem}

\section{The Main Result}\label{section: Existence and Uniqueness}

Our proof for the existence of a mild solution relies on an iterative method which in each step requires solving a deterministic equation, i.e. an equation in which $\omega$ appears only as a parameter. The following theorem proved in Zangeneh~\cite{Zangeneh-measurability} and~\cite{Zangeneh-Thesis} guarantees the solvability of such equations and the measurability of the solution with respect to parameter.

Let $(\Omega,\mathcal{F},\mathcal{F}_t,\mathbb{P})$ be a filtered probability space and assume $f$ satisfies Hypothesis~\ref{main_hypothesis}-(a) and there exists a constant $D$ such that $\|f(t,x,\omega)\|^2 \le D (1+\|x\|^2)$ and assume $V(t,\omega)$ is an adapted process with c\`adl\`ag trajectories and $X_0(\omega)$ is $\mathcal{F}_0$ measurable.

\begin{theorem}[Zangeneh,~\cite{Zangeneh-measurability} and~\cite{Zangeneh-Thesis}]\label{theorem: measurability}
    With assumptions made above, the equation
        \[ X_t=S_t X_0 + \int_0^t S_{t-s}f(s,X_s,\omega) ds + V(t,\omega)\]
    has a unique measurable adapted c\`adl\`ag solution $X_t(\omega)$. Furtheremore
    \[	\|X(t)\| \le \|X_0\|+\|V(t)\|+\int_0^t e^{(\alpha+M)(t-s)} \|f(s,S_s X_0+V(s))\| ds. \]
\end{theorem}

\begin{remark}
    Note that the original theorem is stated for evolution operators and requires some additional assumptions, but those are automatically satisfied for $C_0$ semigroups. (See Curtain and Pritchard~\cite{Curtain-Pritchard} page 29, Theorem 2.21).
\end{remark}

\begin{theorem}[Existence and Uniqueness of the Mild Solution]\label{theorem:existence and uniqueness}
    Under the assumptions of Hypothesis~\ref{main_hypothesis}, equation~\eqref{main_equation} has a unique square integrable c\`adl\`ag mild solution with initial condition $X_0$.
\end{theorem}

This theorem has been stated without proof in~\cite{Proceedings}

    \begin{lemma}\label{lemma: alpha=0}
        It suffices to prove theorem~\ref{theorem:existence and uniqueness} for the case that $\alpha=0$.
    \end{lemma}
    \begin{proof} Define
        \begin{gather*}
           \tilde{S}_t= e^{-\alpha t} S_t ,\qquad \tilde{f}(t,x,\omega)=e^{-\alpha t}f(t,e^{\alpha t}x,\omega) ,\qquad \tilde{g}(t,x,\omega)=e^{-\alpha t}g(t,e^{\alpha t}x,\omega), \\
           \tilde{k}(t,\xi,x,\omega)=e^{-\alpha t}k(t,\xi,e^{\alpha t}x,\omega).
        \end{gather*}
        Note that $\tilde{S}_t$ is a contraction semigroup. It is easy to see that $X_t$ is a mild solution of equation~\eqref{main_equation} if and only if $\tilde{X}_t=e^{-\alpha t} X_t$ is a mild solution of equation with coefficients $\tilde{S},\tilde{f},\tilde{g},\tilde{k}$.
    \end{proof}

\begin{proof}[Proof of Theorem~\ref{theorem:existence and uniqueness}.]
    \emph{Uniqueness.}
    According to the lemma, we can assume $\alpha=0$. Assume that $X_t$ and $Y_t$ are two mild solutions with same initial conditions. Subtracting them we find
    \[ X_t-Y_t=\int_0^t S_{t-s} dZ_s,\]
    where
    \begin{multline*}
        dZ_t=(f(t,X_t)-f(t,Y_t))dt+(g(t,X_{t-})-g(t,Y_{t-}))dW_t\\
        +\int_E{(k(t,\xi,X_{t-})-k(t,\xi,Y_{t-}))d\tilde{N}}.
    \end{multline*}
    Applying It\^o type inequality (Theorem~\ref{theorem:ito type inequality}) for $\alpha=0$ to $X_t-Y_t$ we find
    \[ \| X_t-Y_t \| ^2 \le 2 \int_0^t {\langle X_{s-}-Y_{s-} , d Z_s \rangle}+[Z]_t. \]
    Taking expectations and noting that integrals with respect to cylindrical Wiener processes and compensated Poisson random measures are martingales, we find that
    \[ \mathbb{E}\lVert X_t-Y_t \rVert ^2 \le 2 \int_0^t {\mathbb{E}{\langle X_{s-}-Y_{s-} , f(s,X_s)-f(s,Y_s) \rangle} ds}+\mathbb{E}[Z]_t, \]
    where
    \[ \mathbb{E}[Z]_t = \int_0^t\mathbb{E}\|g(s,X_s)-g(s,Y_s)\|^2ds+\int_0^t\int_E{\mathbb{E}\|k(s,\xi,X_s)-k(s,\xi,Y_s)\|^2\nu(d\xi)ds}. \]
            Note that for a c\`adl\`ag function the set of discontinuity points is countable, hence when integrating with respect to Lebesgue measure, they can be neglected. We therefore neglect the left limits in integrals with respect to the Lebesque measure henceforth. Using assumptions of Hypothesis~\ref{main_hypothesis}-(a) and ~\ref{main_hypothesis}-(b) we find that
    \[ \mathbb{E}\lVert X_t-Y_t \rVert ^2 \le (2M+C) \int_0^t {\mathbb{E}\| X_s-Y_s\|^2 ds}.\]
    Using Gronwall's lemma we conclude that $X_t=Y_t$, almost surely.

    \emph{Existence.}
    It suffices to prove the existence of a solution on a finite interval $[0,T]$. Then one can show easily that these solutions are consistent and give a global solution. We define adapted c\`adl\`ag processes $X^{n}_t$ recursively as follows. Let $X^0_t=S_t X_0$. Assume $X^{n-1}_t$ is defined. Theorem~\ref{theorem: measurability} implies that there exists an adapted c\`adl\`ag solution $X^n_t$ of
    \begin{equation}\label{equation: proof of existence_iteration}
        X^n_t=S_t X_0 + \int_0^t S_{t-s}f(s,X^n_s) ds + V^n_t,
    \end{equation}
    where
        \[ V^{n}_t= \int_0^t{S_{t-s}g(s,X^{n-1}_{s-})d W_s} + \int_0^t{\int_E {S_{t-s}k(s,\xi,X^{n-1}_{s-})} \tilde{N}(ds,d\xi)}. \]
    We wish to show that $\{X^n\}$ converges and the limit is the desired mild solution. This is done by the following lemmas.
    \begin{lemma}\label{lemma: finite_second_moment_iteration}
    	\[ \mathbb{E}\sup\limits_{0\le t\le T} \|X^n_t\|^2<\infty. \]
    \end{lemma}

	\begin{proof}
		We prove by induction on $n$. By Theorem~\ref{theorem: measurability} we have the following estimate,
	    \[	\|X^n_t\|\le \|X_0\|+\|V^n_t\|+\int_0^t e^{M(t-s)} \|f(s,S_s X_0+V^n_s)\| ds. \]
	    Hence,
	    \[	\sup\limits_{0\le t\le T} \|X^n_t\|^2\le 3 \|X_0\|^2+3\sup\limits_{0\le t\le T} \|V^n_t\|^2+3\sup\limits_{0\le t\le T} (\int_0^t e^{M(t-s)} \|f(s,S_s X_0+V^n_s)\| ds)^2, \]
	    where by Cauchy-Schwartz inequality we find
	    \[ \le 3 \|X_0\|^2+3\sup\limits_{0\le t\le T} \|V^n_t\|^2 + 3 T e^{2MT} \int_0^T \|f(s,S_s X_0+V^n_s)\|^2 ds, \]
	    and by Hypothesis~\ref{main_hypothesis}-(c) we have
	    \begin{multline*}
	    	\le 3\|X_0\|^2+3\sup\limits_{0\le t\le T} \|V^n_t\|^2 + 3 T e^{2MT} \int_0^T D(1+\|S_s X_0+V^n_s\|^2) ds \\
	    		\le 3\|X_0\|^2+3\sup\limits_{0\le t\le T} \|V^n_t\|^2 + 3 D T e^{2MT}\int_0^T (1+ 2\|X_0\|^2+2\|V^n_s\|^2) ds\\
	    		= 3DT^2e^{2MT} + (3+6DT^2e^{2MT}) \|X_0\|^2 + (3+6DT^2e^{2MT})\sup\limits_{0\le t\le T} \|V^n_t\|^2.
	    \end{multline*}
	    Hence for completing the proof it suffices to show that $\mathbb{E}\sup\limits_{0\le t\le T} \|V^n_t\|^2<\infty$.

	    Applying Theorem 1.1 of~\cite{Kotelenez-1984} we find,
	    \[ \mathbb{E}\sup\limits_{0\le t\le T} \|V^n_t\|^2 \le \mathbf{C} \mathbb{E} \left( \int_0^T \|g(s,X^{n-1}_{s})\|_{HS}^2 ds +  \int_0^T \int_E \|k(s,\xi,X^{n-1}_{s})\|^2 \nu(d\xi) ds \right)\]
	    where by Hypothesis~\ref{main_hypothesis}-(c),
	    \[ \le \mathbf{C} \mathbb{E} \int_0^T D(1+\|X^{n-1}_s\|^2) ds) \]
	    which is finite by induction Hypothesis.
		The basis of induction follows directly from Hypothesis~\ref{main_hypothesis}-(d).
	\end{proof}

    \begin{lemma}\label{lemma: proof of existence_supremum convergence}
    For $0\le t\le T$ we have,
        \begin{equation} \label{equation: proof of existence_supremum convergence}
            \mathbb{E} \sup\limits_{0\le s\le t} \|X^{n+1}_s-X^n_s\|^2 \le C_0 C_1^n \frac{t^n}{n!}
        \end{equation}
				where $C_1= 2C(1+2\mathcal{C}_1^2) e^{4MT}$ and $C_0=\mathbb{E}\sup\limits_{0\le s\le T} \|X^1_s-X^0_s\|^2$.
				(Note that by Lemma~\ref{lemma: finite_second_moment_iteration}, $C_0<\infty$.)
    \end{lemma}

    \begin{proof}
        We prove by induction on $n$. The statement is obvious for $n=0$. Assume that the statement is proved for $n-1$. We have,
        \begin{equation}\label{equation: proof of existence_X^n+1-X^n}
            X^{n+1}_t-X^n_t= \int_0^t S_{t-s}(f(s,X^{n+1}_s)-f(s,X^n_s)) ds + \int_0^t S_{t-s} dM_s,
        \end{equation}
        where
        \begin{eqnarray*}
            M_t &=& \int_0^t (g(s,X^{n}_{s-})-g(s,X^{n-1}_{s-}))dW_s \\
            & & + \int_0^t \int_E {(k(s,\xi,X^{n}_{s-})-k(s,\xi,X^{n-1}_{s-}))\tilde{N}(ds,d\xi)}.
        \end{eqnarray*}
        Applying It\^o type inequality (Theorem~\ref{theorem:ito type inequality}), for $\alpha=0$, we have
        \begin{multline}\label{equation:proof of existence 1}
            \lVert X^{n+1}_t-X^n_t \rVert ^2 \le 2 \underbrace{\int_0^t {\langle X^{n+1}_{s-}-X^n_{s-},f(s,X^{n+1}_s)-f(s,X^n_s)\rangle ds}}_{A_t}\\
            + 2 \underbrace{\int_0^t {\langle X^{n+1}_{s-}-X^n_{s-},dM_s\rangle}}_{B_t} + [M]_t.
        \end{multline}
				 For the term $A_t$, the semimonotonicity assumption on $f$ implies
        \begin{equation}\label{equation:proof of existence_A_t}
            A_t \le M \int_0^t \|X^{n+1}_s-X^n_s\|^2 ds
        \end{equation}
        We also have
        \begin{multline*}
           \mathbb{E}[M]_t = \int_0^t \mathbb{E}\|g(s,X^{n}_s)-g(s,X^{n-1}_s)\|^2ds\\
           + \int_0^t \int_E \mathbb{E}\|k(s,\xi,X^n_s)-k(s,\xi,X^{n-1}_s)\|^2 \nu(d\xi) ds,
        \end{multline*}
        where by Hypothesis~\ref{main_hypothesis}-(b),
        \begin{equation}\label{equation:proof of existence_[M]_t}
            \le C \int_0^t \mathbb{E}\|X^n_s-X^{n-1}_s\|^2 ds.
        \end{equation}

        Applying Burkholder-Davies-Gundy inequality~(\cite{Peszat-Zabczyk},Theorem 3.50) , for $p=1$, to term $B_t$ we find,
        \begin{eqnarray*}
            \mathbb{E}\sup\limits_{0\le s\le t}|B_s| & \le & \mathcal{C}_1 \mathbb{E}\left([B]_t^\frac{1}{2}\right) \\
            & \le & \mathcal{C}_1 \mathbb{E}\left(\sup\limits_{0\le s\le t} (\|X^{n+1}_s-X^n_s\|) [M]_t^\frac{1}{2} \right)
        \end{eqnarray*}
        where $\mathcal{C}_1$ is the universal constant in the Burkholder-Davies-Gundy inequality. Applying Cauchy-Schwartz inequality we find,
        \begin{equation}\label{equation:proof of existence_B_t}
           \le \frac{1}{4} \mathbb{E} \sup\limits_{0\le s\le t} \|X^{n+1}_s-X^n_s\|^2 + \mathcal{C}_1^2 \mathbb{E}[M]_t.
        \end{equation}
        Now, taking supremums and then expectation on both sides of~\eqref{equation:proof of existence 1} and substituting~\eqref{equation:proof of existence_A_t},~\eqref{equation:proof of existence_[M]_t} and~\eqref{equation:proof of existence_B_t}, we find
        \begin{multline}\label{equation:proof of existence 3}
            \mathbb{E} \sup\limits_{0\le s\le t} \|X^{n+1}_s-X^n_s\|^2 \le 2M \int_0^t \mathbb{E}\|X^{n+1}_s-X^n_s\|^2 ds\\
            + C(1+2\mathcal{C}_1^2) \int_0^t \mathbb{E}\|X^{n}_s-X^{n-1}_s\|^2 ds\\
            + \frac{1}{2} \mathbb{E}(\sup\limits_{0\le s\le t} \|X^{n+1}_s-X^n_s\|)^{2}.
        \end{multline}

        The last term in the right hand side could be subtracted from the left hand side but for this subtraction to be valid it should be finite which is guaranteed by Lemma~\ref{lemma: finite_second_moment_iteration}. After subtraction we find,
            \[ \mathbb{E} \sup\limits_{0\le s\le t} \|X^{n+1}_s-X^n_s\|^2 \le 4 M \int_0^t \mathbb{E}\|X^{n+1}_s-X^n_s\|^2 ds + 2C(1+2\mathcal{C}_1^2) \int_0^t \mathbb{E}\|X^n_s-X^{n-1}_s\|^2 ds, \]
        Now let $h^n(t)=\mathbb{E}\sup\limits_{0\le s\le t} \|X^{n+1}_s-X^n_s\|^2$. Hence,
        \[ h^n(t)\le 4 M \int_0^t h^n(s) ds + 2C(1+2\mathcal{C}_1^2) \int_0^t h^{n-1}(s)ds \]

         Note that by Lemma~\ref{lemma: finite_second_moment_iteration}, $h^n(t)$ is bounded on $[0,T]$. Hence we can use Gronwall's inequality for $h^n(t)$ and find

\[ h^n(t) \le C_1 \int_0^t h^{n-1}(s) ds \]
where by induction hypothesis,

\[ \le C_1 \int_0^t C_0 C_1^{n-1} \frac{s^{n-1}}{(n-1)!} ds = C_0 C_1^n \frac{t^n}{n!} \]
which completes the proof.

\end{proof}

    Returning to the proof of Theorem~\ref{theorem:existence and uniqueness}, we see that since the right hand side of~\eqref{equation: proof of existence_supremum convergence} is a convergent series, $\{X^n\}$ is a cauchy sequence in $L^2(\Omega,\mathcal{F},\mathbb{P};L^\infty([0,T];H))$ and hence converges to a process $X_t(\omega)$. By choosing a subsequence they converge almost sure uniformly with respect to $t$, and since $\{X^n_t\}$ are adapted c\`adl\`ag, so is $X_t$.

    It remains to show that $X_t$ is a solution of~\eqref{mild_solution}. It suffices to show that the terms on both sides of equation~\eqref{equation: proof of existence_iteration} converge to corresponding terms of~\eqref{mild_solution}. We know already that $X^n_t \to X_t$ in $L^2([0,T]\times\Omega;H)$. Moreover by Theorem 1.1 of~\cite{Kotelenez-1984} we have,
    \begin{multline*}
        \mathbb{E}\|\int_0^t S_{t-s} g(s,X^n_{s-})dW_s - \int_0^t S_{t-s} g(s,X_{s-})dW_s\|^2\\
        \le \mathbf{C} \mathbb{E} \int_0^t\|g(s,X^n_s)-g(s,X_s)\|^2 ds\\
        \le \mathbf{C} C \int_0^t\mathbb{E} \|X^n_s-X_s\|ds \to 0,
    \end{multline*}
    and
    \begin{multline*}
        \mathbb{E}\|\int_0^t\int_E S_{t-s} k(s,\xi,X^n_{s-})d\tilde{N} - \int_0^t\int_E S_{t-s} k(s,\xi,X_{s-})d\tilde{N}\|^2\\
        \le \mathbf{C} \mathbb{E} \int_0^t\int_E\|k(s,\xi,X^n_s)-k(s,\xi,X_s)\|^2 \nu(d\xi)ds\\
        \le \mathbf{C} C \int_0^t\mathbb{E} \|X^n_s-X_s\|ds \to 0.
    \end{multline*}
    Hence the terms of $V^n_t$ converge to corresponding terms of~\eqref{mild_solution}. Finally we show that the term containing $f$ in~\eqref{equation: proof of existence_iteration} converges in the weak sense to corresponding term in~\eqref{mild_solution}. If $x\in H$,
    \begin{equation}\label{equation:proof of existence 5}
        \mathbb{E}\langle x,\int_0^t S_{t-s} (f(s,X^n_s)-f(s,X_s)) ds\rangle = \mathbb{E}\int_0^t \langle S_{t-s}^*x,f(s,X^n_s)-f(s,X_s)\rangle ds
    \end{equation}
    By demicontinuity of $f$, the integrand on the right hand side converges to $0$ for almost every $(s,\omega)\in [0,t]\times\Omega$. On the other hand, by Hypothesis~\ref{main_hypothesis}-(c), the integrand is dominated by a constant multiple of $\|x\| (1+\|X_s\|+\|X^n_s\|)$ where $\|X^n_s\| \to \|X_s\|$ pointwise almost everywhere and in $L^1([0,T]\times\Omega)$, hence by dominated convergence theorem we conclude that right hand side of~\eqref{equation:proof of existence 5} tends to $0$. Hence $X_t$ is a mild solution of~\eqref{main_equation}.
\end{proof}

\section{Some Examples}\label{section:examples}

In this section we provide some concrete examples of semilinear stochastic evolution equations with monotone nonlinearity and multiplicative Poisson noise. The examples consist of stochastic partial differential equations of parabolic and hyperbolic type and a stochastic delay differential equation. We show that these examples satisfy the assumptions of equation~\eqref{main_equation} and hence one can apply Theorem~\ref{theorem:existence and uniqueness} to them.

\begin{example}[Stochastic reaction-diffusion equations with multiplicative Poisson noise]
\label{example: finite_diemnsional_noise}
    In this example we consider a class of semilinear stochastic evolution equations with multiplicative Poisson noise. Let $\mathcal{D}$ be a bounded domain with a smooth boundary in $\mathbb{R}^d$. Consider the equation,

    \begin{equation}\label{equation: example_finite_dimensional_noise}
        \left\{\begin{array}{lll}
            d u (t)  &=& Au(t) dt + f(u(t,x)) dt + \eta u(t) dt + \int_E k(t,\xi,u(t^-,x)) \tilde{N}(dt,d\xi)\\
            u(0) & = & u_0.
        \end{array} \right.
    \end{equation}
	where $A$ is the generator of a $C_0$ semigroup on $L^2(\mathcal{D})$, $f:\mathbb{R}\to\mathbb{R}$ is a continuous decreasing function with linear growth and $k:[0,T]\times E \times \mathbb{R}\times\Omega\to \mathbb{R}$ is measurable and satisfies the Lipschitz condition
	\[ \mathbb{E} \int_E |k(s,\xi,u)-k(s,\xi,v)|^2 \mu(d\xi) \le C |u-v|^2 \]
and the linear growth condition
	\[ \mathbb{E} \int_E |k(s,\xi,u)|^2 \mu(d\xi) \le D (1+|u|^2) \]
and $u_0\in L^2(\mathcal{D})$. We show that equation~\eqref{equation: example_finite_dimensional_noise} satisfies the assumptions of equation~\eqref{main_equation}. Let $H=L^2(\mathcal{D})$. We denote the Nemitsky operator associated with a function $f:\mathbb{R}\to\mathbb{R}$ by the same symbol. Since $f$ and $k$ are continuous and have linear growth, by Theorem (2.1) of Krasnosel'ski\u\i~\cite{Krasnoelskii}, the associated Nemitsky operators define continuous operators from $L^2(\mathcal{D})$ to $L^2(\mathcal{D})$ and have linear growth. Verifying the other assumptions is straight forward.

\begin{remark}
	Equation~\eqref{equation: example_finite_dimensional_noise} is exactly the same as the main equation studied in~\cite{Marinelli-Rockner-wellposedness}.
\end{remark}

\end{example}

\begin{example}[Second Order Stochastic Hyperbolic Equations with L\'evy noise] \label{example: finite_diemnsional_noise_hyperbolic}

	In this example we consider a hyperbolic SPDE with L\'evy noise. Let $\mathcal{D}$ be a bounded domain with a smooth boundary in $\mathbb{R}^d$, Consider the initial boundary value problem,
    \begin{equation}\label{equation: example_finite_dimensional_noise_hyperbolic}
        \left\{\begin{array}{lrll}
            \frac{\partial^2 u}{\partial t^2}  = \Delta u -\sqrt[3]{\frac{\partial u}{\partial t}} & + u(t^-,x) \frac{\partial Z}{\partial t}& \textrm{on} & [0,\infty) \times \mathcal{D}\\
            u =0 && \textrm{on} & [0,\infty) \times \partial \mathcal{D}\\
            u(0,x) = u_0(x) && \textrm{on} & \mathcal{D}.\\
            \frac{\partial u}{\partial t} (0,x) = 0 && \textrm{on} & \mathcal{D}.
        \end{array} \right.
    \end{equation}
    where $Z(t)$ is a real valued square integrable L\'evy process and $u_0(x) \in L^2(\mathcal{D})$ is the initial condition. One can replace $-\sqrt[3]{x}$ by any continuous decreasing real function with linear growth.

	Let $H^1(\mathcal{D})$ be the Sobolev space of weakly differentiable functions on $\mathcal{D}$ with derivative in $L^2(\mathcal{D})$ and let $H=H^1(\mathcal{D})\times L^2(\mathcal{D})$.
	
	Note that $\Delta$ is self adjoint and negative definite on $L^2$. Moreover, we have
        \[ D((-\Delta)^\frac{1}{2})=  H^1(\mathcal{D}). \]
    Hence by Lemma B.3 of~\cite{Peszat-Zabczyk}, the operator
        \[ \mathcal{A}=\left( {\begin{array}{cc} 0&I\\  \Delta &0 \end{array}} \right)\]
    generates a $C_0$ semigroup of contractions on $H$.

    Let $K=E=\mathbb{R}$. We also define for $(u,v) \in H$ and $\phi\in K$ and $\xi\in E$,
        \[ f(u,v)= \left( {\begin{array}{c} 0\\ -\sqrt[3]{v(x)} \end{array}} \right), g(u,v)=0 , k(\xi,u,v)= \left( {\begin{array}{c} 0\\ u(x) \xi \end{array}} \right) \]

    Hence equation~\eqref{equation: example_finite_dimensional_noise_hyperbolic} can be written as
        \[dX(t)=\mathcal{A}X(t) dt+ f(X(t))dt+ g(X(t^-))dW_t + \int_E k(\xi,X(t^-)) \tilde{N}(dt,d\xi) \]
	
	We claim that $f$, $g$ and $k$ satisfy Hypothesis~\ref{main_hypothesis}. The continuity of $f$, $g$ and $k$ follow as in example~\ref{example: finite_diemnsional_noise}. The other conditions are straightforward.
\end{example}

\begin{example}[Stochastic Delay Equations]
    In this example we consider a stochastic delay differential equation in $\mathbb{R}$.         The case of Lipschitz coefficients, have been studied before in~\cite{Peszat-Zabczyk}. We have replaced Lipschitzness of $f$ by the weaker assumption of semimonotonicity.

Consider the following equation,
    \begin{equation}\label{equation: example_delay_0}
        \left\{\begin{array}{ll}
            dx(t)=&\left( \int_{-1}^0 x(t+\theta) d\theta\right)dt - \sqrt[3]{x(t)}dt + x(t) dZ_t \\
            x(\theta)=& \sin(\pi \theta) ,\quad \theta \in (-1,0].
          \end{array} \right.
    \end{equation}
	where $Z_t$ is a real valued square integrable L\'evy process. We show that this equation satisfies the assumptions of equation~\eqref{main_equation}. $- \sqrt[3]{x}$ can be replaced by any continuous decreasing real function with linear growth and the initial condition can be replace with any function in $L^2((-1,0])$.

    Let $H=\mathbb{R}\times L^2((-1,0])$ and define the operator $A$ on $H$ by
        \[ A \left( {\begin{array}{c} u \\ v \end{array}} \right) = \left( {\begin{array}{c}  \int_{-1}^0 v(\theta) d\theta \\ \frac{\partial v}{\partial \theta} \end{array}} \right). \]
    According to Da Prato and Zabczyk~\cite{DaPrato_Zabczyk_book}, Proposition A.25, the operator $A$ with domain
        \[ D(A)=\left\{ \left( {\begin{array}{c} u \\ v \end{array}} \right) \in H : v\in W^{1,2}(-1,0), v(0)=u \right\} \]
    generates a $C_0$ semigroup $S_t$ on $H$. Let $K=E=\mathbb{R}$ and let $\tilde{N}$ be the compensated Poisson random measure associated with $Z_t$. Define for $\left( {\begin{array}{c} u \\ v \end{array}} \right)\in H$ and $\xi\in \mathbb{R}$,
        \[ f(u,v)= \left( {\begin{array}{c} - \sqrt[3]{u} \\ 0 \end{array}} \right), g(u,v)=0, k(\xi,u,v)= \left( {\begin{array}{c} \xi  u \\ 0 \end{array}} \right).\]
    It is easy to verify that $f$, $g$ and $k$ satisfy Hypothesis~\ref{main_hypothesis}. Now, if we let
        \[X(t)= \left( {\begin{array}{c} x(t)\\ x_t \end{array}} \right) \]
    where $x_t(\theta)=x(t+\theta)$ for $\theta\in (-1,0]$, then equation~\eqref{equation: example_delay_0} can be written as
        \[dX(t)=A X(t) dt+ f(X(t))dt+ g(X(t^-))dW_t + \int_E k(\xi, X(t^-)) \tilde{N}(dt,d\xi) \]
    with initial condition
        \[ X(0)= \left( {\begin{array}{c} \psi(0) \\ \psi \end{array}} \right) \].
\end{example}

\end{document}